\documentclass{amsart}
\usepackage[utf8]{inputenc}
\usepackage{latexsym,amssymb,amsmath,amsthm,amscd,graphicx}
\usepackage{amsfonts}
\usepackage{xcolor}
\usepackage[english]{babel}
\usepackage[a4paper,bindingoffset=0.2in,%
            left=1in,right=1in,top=1in,bottom=1in,%
            footskip=.25in]{geometry}
\newtheorem{theorem}{Theorem}[section]

\begin{document}

\title[Generalized gometric constants for Morrey spaces]{Some Generalized Geometric Constants for Discrete Morrey Spaces}

\author[H. Rahman]{Hairur Rahman}
\address{Department of Mathematics, Islamic State University of Maulana Malik Ibrahim, Malang 65144, Indonesia}
\email{hairur@mat.uin-malang.ac.id}

\author[H.~Gunawan]{Hendra Gunawan}
\address{Faculty of Mathematics and Natural Sciences, Bandung Institute of Technology, Bandung 40132, Indonesia}
\email{hgunawan@math.itb.ac.id}

\subjclass[2010]{46B20}

\keywords{Generalized Von Neumann-Jordan constant, Zb\"{a}ganu constant, discrete Morrey spaces}

\begin{abstract}
In this paper, we calculate four geometric constants for discrete Morrey spaces.
The constants are generalized von Neumann-Jordan constant, modified von Neumann-Jordan constant,
von Neumann-Jordan type constant, and Zb\"{a}ganu constant. The four constants measure
uniformly nonsquareness of the above spaces. We obtain that the value of each of the four
constants for the above spaces is two, which means that the spaces are NOT uniformly nonsquare.
\end{abstract}

\maketitle

\section{Introduction}

Let $m\in \mathbb{Z}^d := (m_1, \cdots , m_d), N\in \omega := \mathbb{N}\cup {0}$, and write
$$
S_{m,N} :=\lbrace k\in \mathbb{Z}^d : \Vert k-m\Vert_{\infty} \leq N \rbrace,
$$
where $\Vert (m_1, \cdots ,m_d)\Vert_{\infty} := \max \lbrace \vert m_j\vert :1\leq j \leq d \rbrace$
for $(m_1,\cdots , m_d)\in \mathbb{Z}^d$. Then, for each $m\in \mathbb{Z}^d$ and $N\in \omega$,
$\vert S_{m,N}\vert = (2N+1)^d$ denotes the cardinality of $S_{m,N}$.
Let $1\leq p\leq q < \infty$, and define \textit{discrete Morrey spaces} $\ell^p_q (\mathbb{Z}^d)$
as the set of all functions (sequences) $x: \mathbb{Z}^d \to \mathbb{R},$ such that
$$
\Vert x\Vert_{\ell^p_q} := \sup_{m\in \mathbb{Z}^d, N\in \omega} \vert S_{m,N}\vert^{\frac{1}{q}-\frac{1}{p}}
\left( \sum_{k\in S_{m,N}} \vert x(k) \vert^p\right)^{\frac{1}{p}} <\infty.
$$
Note that the space $\ell^p_q$ equipped with the above norm is a Banach space \cite{Gun2}.

Gunawan {\it et al.} \cite{Gun} have computed three geometric constants, namely Von Neumann-Jordan
constant, James constant, and Dunkl-Williams constant, for these discrete Morrey spaces.
We shall here consider other geometric constants known in the literatures. Generalizing the
Von Neumann-Jordan constant for a Banach space $(X,\|\cdot\|_X)$ \cite{Clar}, Cui {\it et al.}
\cite{Cui} define the generalized of von Neumann-Jordan constant $C_{NJ}^{(s)} (X)$ by
$$
C_{NJ}^{(s)} (X) := \sup \left\lbrace \frac{\Vert x+y\Vert_X^s + \Vert x-y\Vert_X^s}{2^{s-1} (\Vert x\Vert_X^s +
\Vert y\Vert_X^s)} : x,y\in X\backslash \lbrace 0\rbrace, \forall s\in [1,\infty)\right\rbrace.
$$
Note that $1\leq C_{NJ}^{(s)} (X) \leq 2$ holds for all Banach spaces $X$. (For $s=1$, $C_{NJ}^s(X)=C_{NJ}(X),$
the Von Nuemann-Jordan constant.)

Meanwhile, Gao \cite{Gao} and Alonso {\it et al.} \cite{Alon} modify the von Neumann-Jordan constant
by defining
$$
C'_{NJ} (X) := \sup \left\lbrace \frac{\Vert x+y\Vert_X^2 + \Vert x-y\Vert_X^2}{4} : x,y\in X,
\Vert x\Vert_X = \Vert y\Vert_X = 1 \right\rbrace.
$$
Here we also have $1\leq C'_{NJ}(X) \leq C_{NJ}(X)\leq 2$ for all Banach spaces $X$.
The modified the von Neumann-Jordan constant is then generalized by Yang {\it et al.} \cite{Yang} to
the following constant:
$$
\bar{C}_{NJ}^{(s)} (X) := \sup \left\lbrace \frac{\Vert x+y\Vert_X^s + \Vert x-y\Vert_X^s}{2^s} :
x,y\in X, \Vert x\Vert_X = \Vert y\Vert_X = 1 \right\rbrace.
$$
One may observe that $\bar{C}_{NJ}^{(s)} (X)\leq C_{NJ}^{(s)} (X) \leq 2^{s-1}\left[1+
\left(2^{\frac{1}{s}}\left(C'_{NJ} (X)\right)^{\frac{1}{s}} -1\right)\right]^{s-1}$ for all Banach spaces $X$.

Another von Neumann-Jordan type constant, denoted by $C_{-\infty}$ is defined by Zuo \cite{Zuo} by
$$
C_{-\infty}(X):= \sup \left\lbrace \frac{\min\left\lbrace\Vert x+y\Vert_X^2  ,\Vert x-y\Vert_X^2\right\rbrace}
{\Vert x\Vert_X^2 +\Vert y\Vert_X^2} : x,y\in X\backslash \lbrace 0\rbrace \right\rbrace .
$$
This constant satisfies the following inequality
$$
C_{-\infty}(X) \leq C_{NJ}(X) \leq C_{J}(X),
$$
where $C_J(X)$ is James constant, given by $C_J(X):=\sup\{\min\{\|x+y\|,\|x-y\|:\|x\|=\|y\|=1\}\}$.
Recently, Zuo \cite{Zuo1} generalizes von Neumann-Jordan type constant to the following constant:
$$
C_{-\infty}^{(s)}(X):= \sup \left\lbrace \frac{\min\left\lbrace\Vert x+y\Vert_X^s  ,\Vert x-y\Vert_X^s
\right\rbrace}{2^{s-2}\left(\Vert x\Vert_X^s +\Vert y\Vert_X^s\right)} : x,y\in X\backslash \lbrace 0\rbrace \right\rbrace ,
$$
for $1\leq s<\infty$. Zuo proves that $X$ is uniformly nonsquare if only if $C_{-\infty}^{(s)} (X) <2$
for some $1\leq s<\infty$. A Banach space $X$ is {\em uniformly nonsquare} if and only if
there exists a $\delta>0$ such that
\[
\min \{\|x+y\|_X,\|x-y\|_X\} \le 2(1-\delta)
\]
for every $x,y\in X$ with $\|x\|_X=\|y\|_X=1$. Consequently, $X$ is uniformly nonsquare if and only if $C_J(X)<2$ \cite{james}.

There is also a constant called Zb\'{a}ganu constant introduced in \cite{Ganu} and studied by \cite{Lor,Lor1}:
$$
C_Z (X) := \sup \left\lbrace \frac{\Vert x+y\Vert_X  \Vert x-y\Vert_X}{\Vert x\Vert_X^2 +\Vert y\Vert_X^2} :
x,y\in X\backslash \lbrace 0\rbrace \right\rbrace.
$$
Clearly $1\leq C_Z(X) \leq C_{NJ}(X)\leq 2$ is satisfied for all normed spaces $X$. Zb\'{a}ganu constant
is generalized by \cite{Zhang}, for $1\leq s<\infty$, to the following constant:
$$C_Z^{(s)} (X):= \sup \left\lbrace \frac{\Vert x+y\Vert_X^{\frac{s}{2}}  \Vert x-y\Vert_X^{\frac{s}{2}}}{2^{s-2}
\left(\Vert x\Vert_X^s +\Vert y\Vert_X^s\right)} : x,y\in X\backslash \lbrace 0\rbrace \right\rbrace.
$$
It is obvious that $2^{2-s}\leq C_{-\infty}^{(s)}(X)\leq C_Z^{(s)}\leq C_{NJ}^{(s)} \leq 2$ for all $1\leq s<\infty$
(see \cite{Zhang}).

All of these constants, except von Neumann-Jordan type constant, have been computed for Morrey spaces and small
Morrey spaces (see \cite{Rah}). In this paper, we are interested in computing all of those constants for
discrete Morrey spaces. We present our results in the following section.

\section{Main Results}

Our results for discrete Morrey spaces are presented in the following theorem:

\bigskip

\begin{theorem}
Let $1\leq p< q<\infty$ and $1\leq s<\infty$. Then we have
$$
C_{NJ}^{(s)} (\ell^p_q) = C'_{NJ}(\ell^p_q) = \bar{C}_{NJ}^{(s)} (\ell^p_q) =
C_{-\infty} (\ell^p_q) = C_{-\infty}^{(s)}(\ell^p_q) = C_Z(\ell^p_q) = C_Z^{(s)}(\ell^p_q)= 2 .
$$
\end{theorem}

\begin{proof} The key to the proof is to find two elements $x,y\in \ell^p_q$ with norms equal to one,
such that $\|x+y\|_{\ell^p_q}=\|x-y\|_{\ell^p_q}=2$. Let us first consider the case where
$d=1$. For $1\leq p<q<\infty$, let $n\in \mathbb{Z}$ be an even number with $n>2^{\frac{q}{q-p}} -1$,
which can be written as
$$
(n+1)^{\frac{1}{q} - \frac{1}{p}} < 2^{-\frac{1}{p}} .
$$
Let us define the sequence $(x_k)_{k\in \mathbb{Z}}$ by
$$
x_0 = x_n =1, \text{ and } x_k =0 \text{ for all } k\not\in \lbrace 0,n \rbrace
$$
and the sequence $(y_k)_{k\in \mathbb{Z}}$ by
$$
y_0 = 1, y_n =-1, \text{ and } y_k =0 \text{ for all } k\not\in \lbrace 0,n \rbrace.
$$
Thus, we have
\begin{align*}
\Vert x\Vert_{\ell^p_q} &= \sup_{m\in \mathbb{Z}^d, N\in \omega} \vert S_{m,N}\vert^{\frac{1}{q}-\frac{1}{p}}
\left( \sum_{k\in S_{m,N}} \vert x_k \vert^p\right)^{\frac{1}{p}}\\
&= \max \left\lbrace 1, \vert S_{\frac{n}{2},\frac{n}{2}}\vert^{\frac{1}{q}-\frac{1}{p}}
\left( \sum_{k\in S_{\frac{n}{2},\frac{n}{2}}} \vert x_k \vert^p\right)^{\frac{1}{p}} \right\rbrace \\
&= \max \left\lbrace 1, (n+1)^{\frac{1}{q} - \frac{1}{p}} 2^{\frac{1}{p}} \right\rbrace.
\end{align*}
By choosing $n$ as above, we can see that
$$
(n+1)^{\frac{1}{q} - \frac{1}{p}} 2^{\frac{1}{p}} < 1.
$$
Therefore $\Vert x\Vert_{\ell^p_q} = 1$. By the same argument, we can verify that $\Vert y\Vert_{\ell^p_q} = 1$.
Furthermore, we may observe that
$$
\Vert x+y \Vert_{\ell^p_q} = 2 \text{ and } \Vert x-y\Vert_{\ell^p_q} = 2.
$$

Now we shall consider the general case where $d\geq 1$. Let $n\in \mathbb{Z}$ be an even number with
$n>2^{\frac{q}{d(q-p)}} -1$, which can be written as
$$
(n+1)^{d\left(\frac{1}{q} - \frac{1}{p}\right)} < 2^{-\frac{1}{p}} .
$$
Let us define the function $x\in \ell^p_q$ where
\begin{equation*}
x(k) :=
\begin{cases}
1, \text{if } k=(0,0,\cdots , 0), (n,0,\cdots , 0)\\
0, \text{otherwise}
\end{cases}
\end{equation*}
and $y \in \ell^p_q$ where
\begin{equation*}
y(k) :=
\begin{cases}
1, \text{if } k=(0,0,\cdots , 0)\\
-1, \text{if } k=(n,0,\cdots , 0)\\
0, \text{otherwise}
\end{cases}
\end{equation*}
Thus, we have
\begin{align*}
\Vert x\Vert_{\ell^p_q} &= \sup_{m\in \mathbb{Z}^d, N\in \omega} \vert S_{m,N}\vert^{\frac{1}{q}-\frac{1}{p}}
\left( \sum_{k\in S_{m,N}} \vert x_k \vert^p\right)^{\frac{1}{p}}\\
&= \max \left\lbrace 1, \vert S_{\frac{n}{2},\frac{n}{2}}\vert^{d\left(\frac{1}{q}-\frac{1}{p}\right)}
\left( \sum_{k\in S_{\frac{n}{2},\frac{n}{2}}} \vert x_k \vert^p\right)^{\frac{1}{p}} \right\rbrace \\
&= \max \left\lbrace 1, (n+1)^{d\left(\frac{1}{q} - \frac{1}{p}\right)} 2^{\frac{1}{p}} \right\rbrace.
\end{align*}
With the choice of $n$ as above, we have
$$
(n+1)^{d\left(\frac{1}{q} - \frac{1}{p}\right)} 2^{\frac{1}{p}} < 1.
$$
Therefore $\Vert x\Vert_{\ell^p_q} = 1$, and by the same argument $\Vert y\Vert_{\ell^p_q} = 1$.
Furthermore, we also have
$$
\Vert x+y \Vert_{\ell^p_q} = 2 \text{ and } \Vert x-y\Vert_{\ell^p_q} = 2.
$$

Now, we can calculate for generalized von Neumann-Jordan constant as follows
\begin{equation*}
C_{NJ}^{(s)} \left(\ell^p_q\right) \geq \frac{\Vert x+y\Vert^s_{\ell^p_q} + \Vert x-y\Vert^s_{\ell^p_q}}{2^{s-1}
(\Vert x\Vert^s_{\ell^p_q} +\Vert y\Vert^s_{\ell^p_q})} = \frac{2^{s+1}}{2^s} = 2.
\end{equation*}
By this observation, we find that $C_{NJ}^{(s)} \left(\ell^p_q\right) \geq 2$.
Since $C_{NJ}^{(s)} \left(\ell^p_q\right) \leq 2$, we can conclude that $C_{NJ}^{(s)} \left(\ell^p_q\right) = 2$.

For the same elements $x$ and $y$, we have
\begin{equation*}
C'_{NJ} \left(\ell^p_q\right) \geq \frac{\Vert x+y\Vert_{\ell^p_q}^2 + \Vert x-y\Vert_{\ell^p_q}^2}{4} = \frac{2^2 + 2^2}{4} = 2.
\end{equation*}
Using similiar arguments, we conclude that $C'_{NJ} \left(\ell^p_q\right) = 2$.
As for the generalized of modified von Neumann-Jordan constant, we obtain
\begin{equation*}
\bar{C}_{NJ}^{(s)} \left(\ell^p_q\right) \geq \frac{\Vert x+y\Vert_{\ell^p_q}^s + \Vert x-y\Vert_{\ell^p_q}^s}{2^s} =
\frac{2^{s+1}}{2^s} = 2.
\end{equation*}
Using the same argument as above, we conclude that $\bar{C}_{NJ}^{(s)} \left(\ell^p_q\right) = 2$.

Next, we have von Neumann-Jordan type constant with the same choices for $x$ and $y$, namely
\begin{equation*}
C_{-\infty}\left(\ell^p_q\right) = \sup \left\lbrace \frac{\min\left\lbrace\Vert x+y\Vert_{\ell^p_q}^2,
\Vert x-y\Vert_{\ell^p_q}^2\right\rbrace}{\Vert x\Vert_{\ell^p_q}^2 +\Vert y\Vert_{\ell^p_q}^2} :
x,y\in \ell^p_q\backslash \lbrace 0\rbrace \right\rbrace = 2.
\end{equation*}
Also for generalized von Neumann-Jordan type constant, we observe that
\begin{equation*}
C_{-\infty}^{(s)}\left(\ell^p_q\right)\geq  \frac{\min\left\lbrace\Vert x+y\Vert_{\ell^p_q}^s,
\Vert x-y\Vert_{\ell^p_q}^s\right\rbrace}{2^{s-2}\left(\Vert x\Vert_{\ell^p_q}^s +\Vert y\Vert_{\ell^p_q}^s\right)}
= \frac{2^s}{2^{s-1}} = 2.
\end{equation*}
Since $C_{-\infty}^{(s)}\left(\ell^p_q\right)\leq C_{NJ}^{(s)}\left(\ell^p_q\right)= 2$,
it follows that $C_{-\infty}^{(s)}\left(\ell^p_q\right)=2$.

Last, for the Zb\'{a}ganu constant and generalized Zb\'{a}ganu constant, we obtain that
\begin{equation*}
C_Z \left(\ell^p_q\right) \geq \frac{\Vert x+y\Vert_{\ell^p_q} \Vert x-y\Vert_{\ell^p_q}}{\Vert x\Vert_{\ell^p_q}^2
+\Vert y\Vert_{\ell^p_q}^2} = \frac{4}{2} = 2
\end{equation*}
and
\begin{equation*}
C_Z^{(s)} \left(\ell^p_q\right) \geq \frac{\Vert x+y\Vert_{\ell^p_q}^{\frac{s}{2}}
\Vert x-y\Vert_{\ell^p_q}^{\frac{s}{2}}}{2^{s-2}\left(\Vert x\Vert_{\ell^p_q}^s +
\Vert y\Vert_{\ell^p_q}^p\right)} = \frac{2^s}{2^{s-1}} = 2.
\end{equation*}
By the same argument, we conclude that $C_Z \left(\ell^p_q\right) = 2$ and
since $C_{Z}^{(s)}\left(\ell^p_q\right)\leq C_{NJ}^{(s)}\left(\ell^p_q\right)= 2$,
it is also clear that $C_{Z}^{(s)}\left(\ell^p_q\right)=2$.
\end{proof}

\bigskip

\noindent{\bf Acknowledgements}. This research was supported by Departement of Mathematics, Islamic 
State University of Maulana Malik Ibrahim Malang under UIN-Maliki Research Program 2020. The second 
author is supported by ITB Research and Innovation Program 2020.

\end{document}